\documentclass{amsart}

\usepackage{amssymb,amsmath,amsthm,latexsym,booktabs}

\theoremstyle{definition}
\newtheorem{definition}{Definition}[section]

\theoremstyle{plain}
\newtheorem{lemma}[definition]{Lemma}
\newtheorem*{theorem}{Theorem}
\newtheorem*{corollary}{Corollary}

\newcommand{\nn}{\!\!\!\!}

\allowdisplaybreaks

\begin{document}

\title[A polynomial identity in Lie-Yamaguti algebras]
{A polynomial identity for the bilinear operation in Lie-Yamaguti algebras}

\author{Murray R. Bremner}

\address{Department of Mathematics and Statistics, University of Saskatchewan, Canada}

\email{bremner@math.usask.ca}



\thanks{This research was supported by a Discovery Grant from NSERC, 
the Natural Sciences and Engineering Research Council of Canada.
The author thanks Pilar Benito and the Department of Mathematics and Computer Science 
at the University of La Rioja in Logro\~no for their hospitality in April and May 2013.}

\begin{abstract}
We use computer algebra to demonstrate the existence of a multilinear polynomial identity of degree 8 satisfied 
by the bilinear operation in every Lie-Yamaguti algebra.
This identity is a consequence of the defining identities for Lie-Yamaguti algebras, but is not a consequence of 
anticommutativity.
We give an explicit form of this identity as an alternating sum over all permutations of the variables in a polynomial 
with 8 terms.
Our computations also show that such identities do not exist in degrees less than 8.
\end{abstract}

\maketitle

\allowdisplaybreaks

%%%%%%%%%%%%%%%%%%%%%%%%%%%%%%%%%%%%%%%%%%%%%%%%%%%%%%%%%%%%%%%%%%%%%%%%%%%%%%

\section{Introduction}

The binary-ternary structures now called Lie-Yamaguti algebras were introduced by Yamaguti \cite{Yamaguti} in his study of 
torsion and curvature on symmetric spaces.  

\begin{definition} \label{LYdefinition}
A vector space $L$ with a bilinear operation 
$[-,-]\colon L \times L \to L$ 
and a trilinear operation 
$\langle -,-,- \rangle\colon L \times L \times L \to L$ is a \textbf{Lie-Yamaguti algebra} if the operations satisfy 
the following polynomial identities for all $a, b, c, d, e \in L$:
  \begin{align}
  &
  [a,a] \equiv 0,
  \label{LY1}
  \\
  &
  \langle a, a, b \rangle \equiv 0,
  \label{LY2}
  \\
  &
  \sum_{a,b,c} \big( \, [[a,b],c] + \langle a, b, c \rangle \big) \equiv 0 \;\; \text{(cyclic sum)},
  \label{LY3}
  \\
  &
  \sum_{a,b,c} \langle \, [a,b], \, c, \, d \, \rangle \equiv 0 \;\; \text{(cyclic sum)},
  \label{LY4}
  \\
  &
  \langle \, a, \, b, \, [c,d] \, \rangle 
  \equiv 
  [ \, \langle a, b, c \rangle, \, d \, ] + 
  [ \, c, \, \langle a, b, d \rangle \, ],
  \label{LY5}
  \\
  &
  \langle \, a, \, b, \, \langle c, d, e \rangle \, \rangle 
  \equiv
  \langle \, \langle a, b, c \rangle, \, d, \, e \, \rangle +
  \langle \, c, \, \langle a, b, d \rangle, \, e \, \rangle +
  \langle \, c, \, d, \, \langle a, b, e \rangle \, \rangle.
  \label{LY6}
  \end{align}
\end{definition}

From this it is clear that Lie-Yamaguti algebras are a generalization of Lie algebras and Lie triple systems: 
if the binary (resp.~ternary) operation is identically zero, we obtain a Lie triple system (resp.~Lie algebra).
Yamaguti called these structures \emph{general Lie triple systems}; the modern name was introduced by Kinyon and Weinstein
\cite{KinyonWeinstein}.

Given two elements $a, b \in L$ we define a linear operator $D_{a,b} \in \mathrm{End}(L)$ by 
$D_{a,b}(c) = \langle a,b,c \rangle$; this is a derivation of both the bilinear and trilinear operations.
We denote by $D \subseteq \mathrm{End}(L)$ the subspace spanned by all $D_{a,b}$ $(a, b \in L)$, and define an 
anticommutative product $\cdot$ on $G = D \oplus L$ as follows:
  \[
  D_{a,b} \cdot D_{c,d} = D_{\langle a,b,c \rangle,d} + D_{c,\langle a,b,d \rangle},
  \quad
  D_{a,b} \cdot c = \langle a,b,c \rangle,
  \quad
  a \cdot b = D_{a,b} + [a,b].
  \]
Yamaguti showed that $G$ is a Lie algebra, the \emph{standard enveloping algebra} of $L$.
Furthermore, $D$ is a Lie subalgebra of $G$, and $D \cdot L \subseteq L$, giving a \emph{reductive decomposition} of $G$.
Finally, we recover the bilinear and trilinear operations on $L$ from the product in $G$ by the formulas
$[a,b] = p_L( a \cdot b )$ and $\langle a,b,c \rangle = p_D( a \cdot b ) \cdot c$ where $p_D$ and $p_L$ are the projections
onto $D$ and $L$.
The structure theory of finite dimensional Lie-Yamaguti algebras has been developed during the last few years by Benito,
Elduque, and Mart\'in-Herce \cite{BEM2009,BEM2011}.

Loday \cite{Loday} introduced the notion of a \emph{Leibniz algebra}, which is a nonassociative algebra $A$ with product
$\{a,b\}$ satisfying the \emph{derivation identity},
  \[
  \{ \{a,b\}, c \} \equiv \{ \{a,c\}, b \} + \{ a, \{b,c\} \}.
  \]
Kinyon and Weinstein showed that the skew-symmetrization $[a,b] = \{a,b\}-\{b,a\}$ of the product on any Leibniz algebra 
can be realized as the bilinear operation on a Lie-Yamaguti algebra, where the trilinear operation is 
$\langle a,b,c \rangle = -\frac14 \{ \{ a, b \}, c \}$.

In spite of all this, the bilinear operation in a Lie-Yamaguti algebra remains somewhat mysterious; the only 
identity in Definition \ref{LYdefinition} which involves the bilinear operation by itself is anticommutativity \eqref{LY1}.
This raises the following question: 
  \begin{quote}
\emph{Let $LY(X)$ be the free Lie-Yamaguti algebra on the set $X$ of generators, and let $BLY(X)$ be the subalgebra 
generated by $X$ using only the bilinear operation.  Is $BLY(X)$ a free anticommutative algebra?}
  \end{quote}
(For free anticommutative algebras, see Shirshov \cite{Shirshov}.)
This question can be reformulated in terms of polynomial identities as follows:
  \begin{quote}
\emph{Do there exist polynomial identities satisfied by the operation $[-,-]$ in every Lie-Yamaguti algebra which follow 
from the defining identities \eqref{LY1}--\eqref{LY6} but do not follow from anticommutativity \eqref{LY1} alone?}
  \end{quote}
From the calculations of the author and Hentzel \cite{BremnerHentzel}, it is known that if such an identity exists, it must
have degree at least 7.
That paper studied the anticommutative product on the $\mathfrak{sl}_2(\mathbb{C})$-module $V(10)$ obtained from the 
decomposition
  \[
  \Lambda^2 V(10) \approx V(18) \oplus V(14) \oplus V(10) \oplus V(6) \oplus V(2).
  \]
In this isomorphism, the direct sum of $V(10)$ with the adjoint module $V(2) \cong \mathfrak{sl}_2(\mathbb{C})$ can be 
given the structure of a Lie algebra of type $G_2$, and hence $V(10)$ becomes a Lie-Yamaguti algebra where the operations 
$[a,b] = \beta(a,b)$ and $\langle a, b, c \rangle = \tau(a,b) \cdot c$ are determined by the projections 
$\beta\colon \Lambda^2 V(10) \to V(10)$ and $\tau\colon \Lambda^2 V(10) \to V(2)$.
The simplest identities satisfied by the bilinear operation $\beta(a,b)$ which do not follow from anticommutativity have
degree 7.
Similar calculations for the Lie-Yamaguti structure on $V(6)$ were done by the author and Douglas \cite{BremnerDouglas}.
(For a detailed analysis of Lie-Yamaguti algebras related to $G_2$, see Benito, Draper and Elduque \cite{BDE}.)

In this note, we describe calculations with the computer algebra system Maple that demonstrate the
existence of a polynomial identity in degree~8 which is satisfied by the bilinear operation in every Lie-Yamaguti algebra
but which is not a consequence of anticommutativity. 
We give an explicit statement of this identity, and show that such identities do not exist in degrees less than 8.

%%%%%%%%%%%%%%%%%%%%%%%%%%%%%%%%%%%%%%%%%%%%%%%%%%%%%%%%%%%%%%%%%%%%%%%%%%%%%%

\section{Polynomial identities in Lie-Yamaguti algebras}

We consider binary-ternary algebras, or BT-algebras for short, with a bilinear operation $[-,-]$ and a trilinear operation
$\langle -,-,- \rangle$ satisfying identities \eqref{LY1} and \eqref{LY2}.

\begin{definition}
We write $BT(n)$ for the multilinear subspace of degree $n$ in the free BT-algebra on $n$ generators.
This space is a module over the symmetric group $S_n$ acting by permutations of the subscripts: 
$\sigma \cdot x_i = x_{\sigma(i)}$ for $\sigma \in S_n$ and $1 \le i \le n$.
\end{definition}

We need to determine the number $b t_n$ of inequivalent placements of the binary and ternary operation symbols 
in a monomial of degree $n$ in a free BT-algebra.

\begin{lemma}
The number $bt_n$ of binary-ternary association types in degree $n$ is given by $bt_1 = 1$ 
and the following recurrence relation:
  \begin{align*}
  bt_n 
  &= \!\!\!
  \sum_{i=1}^{\lfloor (n-1)/2 \rfloor} \!\!\! bt_{n-i} bt_i 
  + 
  \binom{bt_{n/2}{+}1}{2}
  \\
  &\quad
  +
  \sum_{i=1}^{n-2} 
  \Bigg[
  \sum_{j=1}^{\lfloor (n-i-1)/2 \rfloor} \!\!\! bt_{n-i-j} bt_j bt_i
  +
  \binom{bt_{(n-i)/2}{+}1}{2} bt_i 
  \Bigg].
  \end{align*}
The terms with binomial coefficients occur only when defined: the first for $n$ even and the second for $n{-}i$ even.
\end{lemma}

In this recurrence relation, the quadratic (resp.~cubic) terms correspond to applications of the bilinear 
(resp.~trilinear) product.
The $bt_n$ association types are partitioned into three subtypes: $b_n$ binary types involve only the bilinear product, 
$t_n$ ternary types involve only the trilinear product, and the remaining $m_n$ mixed types involve both.
We have $bt_n = b_n + t_n + m_n$ for $n \ge 2$.

\medskip

\begin{center}
\begin{tabular}{l|rrrrrrrrrrrr}
$n$ & 1 & 2 & 3 & 4 & 5 & 6 & 7 & 8 & 9 & 10 & 11 & 12 \\
$bt_n$ & 1 & 1 & 2 & 5 & 13 & 38 & 113 & 354 & 1128 & 3688 & 12229 & 41161 \\
\midrule
$b_n$ & 1 & 1 & 1 & 2 & 3 & 6 & 11 & 23 & 46 & 98 & 207 & 451 \\
$t_n$ & 1 & 0 & 1 & 0 & 2 & 0 & 6 & 0 & 19 & 0 & 67 & 0 \\
$m_n$ & 0 & 0 & 0 & 3 & 8 & 32 & 96 & 331 & 1063 & 3590 & 11955 & 40710 \\
\midrule
\end{tabular}
\end{center}

\medskip

We will order the association types in each degree with the ternary and 
mixed types first, and the binary types last. 
Tables \ref{BTtypes5} and \ref{BTtypes6} display the BT types in degrees 5 and 6, 
ordered according to the following definition.

\begin{definition}
If $T$ is a binary-ternary association type then we set $\alpha(T) = \{2\}$ (resp.~$\alpha(T) = 3$) if $T$ involves only 
the binary (resp.~ternary) operation, and $\alpha(T) = \{2,3\}$ if $T$ involves both operations.
We impose the order $\{3\} \prec \{2,3\} \prec \{2\}$. 
If $T = [ U_1, U_2 ]$ (resp.~$T = \langle U_1, U_2, U_3 \rangle$) then we set $\omega(T) = 2$ (resp.~$\omega(T) = 3$).
If $T$ and $T'$ are binary-ternary association types of degrees $n$ and $n'$ then we say that 
$T \prec T'$ in \textbf{deglex order} (degree-lexicographical order) if:
  \begin{itemize}
  \item
  $n < n'$, or
  \item
  $n = n'$ and $\alpha(T) \prec \alpha(T')$, or
  \item
  $n = n'$, $\alpha(T) = \alpha(T')$ and $\omega(T) < \omega(T')$, or
  \item
  $n = n'$, $\alpha(T) = \alpha(T')$, $\omega(T) = \omega(T')$ and $U_i \prec U'_i$, 
  where $T = [U_1,U_2]$, $T' = [U_1,U'_2]$ (or $T = \langle U_1,U_2,U_3 \rangle$, $T' = \langle U'_1,U'_2,U'_3 \rangle$)
  and $i$ is the least index for which $U_i \ne U'_i$.
  \end{itemize}
\end{definition}

  \begin{table}
  \[
  \begin{array}{llll} 
  { \langle - - \langle - - - \rangle \rangle  } &\quad % 1 
  { \langle \langle - - - \rangle - - \rangle  } \\ % 2 
  \midrule
  { [ \langle - - - \rangle [ - - ] ]  } &\quad % 3 
  { [ [ \langle - - - \rangle - ] - ]  } &\quad % 4 
  { [ \langle - - [ - - ] \rangle - ]  } &\quad % 5 
  { [ \langle [ - - ] - - \rangle - ]  } \\ % 6 
  { \langle - - [ [ - - ] - ] \rangle  } &\quad % 7 
  { \langle [ - - ] - [ - - ] \rangle  } &\quad % 8 
  { \langle [ - - ] [ - - ] - \rangle  } &\quad % 9 
  { \langle [ [ - - ] - ] - - \rangle  } \\ % 10 
  \midrule
  { [ [ [ - - ] - ] [ - - ] ]  } &\quad % 11 
  { [ [ [ - - ] [ - - ] ] - ]  } &\quad % 12 
  { [ [ [ [ - - ] - ] - ] - ]  } % 13 
  \end{array}
  \]
  \caption{Binary-ternary association types in degree 5}
  \label{BTtypes5}
  \end{table}
  
  \begin{table}
  \[
  \begin{array}{llll}   
  { [ \langle - - - \rangle \langle - - - \rangle ]  } &\;\; % 1 
  { [ \langle - - - \rangle [ [ - - ] - ] ]  } &\;\; % 2 
  { [ [ \langle - - - \rangle - ] [ - - ] ]  } &\;\; % 3 
  { [ \langle - - [ - - ] \rangle [ - - ] ]  } \\ % 4 
  { [ \langle [ - - ] - - \rangle [ - - ] ]  } &\;\; % 5 
  { [ \langle - - \langle - - - \rangle \rangle - ]  } &\;\; % 6 
  { [ \langle \langle - - - \rangle - - \rangle - ]  } &\;\; % 7 
  { [ [ \langle - - - \rangle [ - - ] ] - ]  } \\ % 8 
  { [ [ [ \langle - - - \rangle - ] - ] - ]  } &\;\; % 9 
  { [ [ \langle - - [ - - ] \rangle - ] - ]  } &\;\; % 10 
  { [ [ \langle [ - - ] - - \rangle - ] - ]  } &\;\; % 11 
  { [ \langle - - [ [ - - ] - ] \rangle - ]  } \\ % 12 
  { [ \langle [ - - ] - [ - - ] \rangle - ]  } &\;\; % 13 
  { [ \langle [ - - ] [ - - ] - \rangle - ]  } &\;\; % 14 
  { [ \langle [ [ - - ] - ] - - \rangle - ]  } &\;\; % 15 
  { \langle - - [ \langle - - - \rangle - ] \rangle  } \\ % 16 
  { \langle - - \langle - - [ - - ] \rangle \rangle  } &\;\; % 17 
  { \langle - - \langle [ - - ] - - \rangle \rangle  } &\;\; % 18 
  { \langle - - [ [ - - ] [ - - ] ] \rangle  } &\;\; % 19 
  { \langle - - [ [ [ - - ] - ] - ] \rangle  } \\ % 20 
  { \langle [ - - ] - \langle - - - \rangle \rangle  } &\;\; % 21 
  { \langle [ - - ] - [ [ - - ] - ] \rangle  } &\;\; % 22 
  { \langle [ - - ] [ - - ] [ - - ] \rangle  } &\;\; % 23 
  { \langle \langle - - - \rangle - [ - - ] \rangle  } \\ % 24 
  { \langle \langle - - - \rangle [ - - ] - \rangle  } &\;\; % 25 
  { \langle [ [ - - ] - ] - [ - - ] \rangle  } &\;\; % 26 
  { \langle [ [ - - ] - ] [ - - ] - \rangle  } &\;\; % 27 
  { \langle [ \langle - - - \rangle - ] - - \rangle  } \\ % 28 
  { \langle \langle - - [ - - ] \rangle - - \rangle  } &\;\; % 29 
  { \langle \langle [ - - ] - - \rangle - - \rangle  } &\;\; % 30 
  { \langle [ [ - - ] [ - - ] ] - - \rangle  } &\;\; % 31 
  { \langle [ [ [ - - ] - ] - ] - - \rangle  } \\ % 32 
  \midrule
  { [ [ [ - - ] - ] [ [ - - ] - ] ]  } &\;\; % 33 
  { [ [ [ - - ] [ - - ] ] [ - - ] ]  } &\;\; % 34 
  { [ [ [ [ - - ] - ] - ] [ - - ] ]  } &\;\; % 35 
  { [ [ [ [ - - ] - ] [ - - ] ] - ]  } \\ % 36 
  { [ [ [ [ - - ] [ - - ] ] - ] - ]  } &\;\; % 37 
  { [ [ [ [ [ - - ] - ] - ] - ] - ]  } &\;\; % 38 
  \end{array}
  \]
  \caption{Binary-ternary association types in degree 6}
  \label{BTtypes6}
  \end{table}
 
\begin{definition}
Let $T$ be a binary-ternary association type in degree $n$.
Let $\iota \in S_n$ be the identity permutation, and consider $\sigma \in S_n$ with $\sigma \ne \iota$ but
$\sigma^2 = \iota$.
Let $[\iota]_T + [\sigma]_T$ be the multilinear polynomial in $BT(n)$ whose two terms have type $T$
with variables $x_1 \cdots x_n$ and $x_{\sigma(1)} \cdots x_{\sigma(n)}$ respectively.
If $[\iota]_T + [\sigma]_T \equiv 0$ in $BT(n)$ then we call this identity a \textbf{skew-symmetry} of association
type $T$.
We write $s(T)$ for the number of skew-symmetrices of association type $T$.
\end{definition}

We will only require the skew-symmetries of the binary association types in each degree.
In degree 8 the 23 binary types are displayed in Table \ref{binary8} with their index numbers as subscripts.
The corresponding 74 skew-symmetries are displayed in Table \ref{skew8}; we give only the index number $t$,
together with $\sigma$ expressed as a product of disjoint transpositions and its sign.

\begin{definition} \label{definitionlifting}
Let $f = f( a_1, \dots, a_n )$ be an element of $BT(n)$: a multilinear polynomial of degree $n$ in the free BT-algebra on $n$
generators.
The bilinear operation produces $n{+}1$ consequences of $f$ in degree $n{+}1$, called the \textbf{binary liftings} of $f$,
using $n$ substitutions and one multiplication:
  \begin{align*}
  &
  f( [a_1,a_{n+1}], a_2, \dots, a_n ), 
  \qquad
  \dots,
  \qquad
  f( a_1, \dots, a_{n-1}, [a_n,a_{n+1}] ),
  \\
  &
  [ f( a_1, \dots, a_n ), a_{n+1} ].
  \end{align*}
The trilinear operation produces $n{+}2$ consequences of $f$ in degree $n{+}2$, called the \textbf{ternary liftings} of $f$, 
using $n$ substitutions and two multiplications:
  \begin{align*}
  &
  f( \langle a_1, a_{n+1}, a_{n+2} \rangle, a_2, \dots, a_n ), 
  \qquad
  \dots,
  \qquad
  f( a_1, \dots, a_{n-1}, \langle a_n, a_{n+1}, a_{n+2} \rangle ),
  \\
  &
  \langle f( a_1, \dots, a_n ), a_{n+1}, a_{n+2} \rangle,
  \qquad
  \langle a_{n+1}, a_{n+2}, f( a_1, \dots, a_n ) \rangle.
  \end{align*}
\end{definition}

  \begin{table}
  \[
  \begin{array}{lll}    
  { [ [ [ - - ] [ - - ] ] [ [ - - ] [ - - ] ] ]_1 } &\;  % 332 
  { [ [ [ - - ] [ - - ] ] [ [ [ - - ] - ] - ] ]_2 } &\;  % 333 
  { [ [ [ [ - - ] - ] - ] [ [ [ - - ] - ] - ] ]_3 } \\ % 334 
  { [ [ [ [ - - ] - ] [ - - ] ] [ [ - - ] - ] ]_4 } &\;  % 335 
  { [ [ [ [ - - ] [ - - ] ] - ] [ [ - - ] - ] ]_5 } &\;  % 336 
  { [ [ [ [ [ - - ] - ] - ] - ] [ [ - - ] - ] ]_6 } \\ % 337 
  { [ [ [ [ - - ] - ] [ [ - - ] - ] ] [ - - ] ]_7 } &\;  % 338 
  { [ [ [ [ - - ] [ - - ] ] [ - - ] ] [ - - ] ]_8 } &\;  % 339 
  { [ [ [ [ [ - - ] - ] - ] [ - - ] ] [ - - ] ]_9 } \\ % 340 
  { [ [ [ [ [ - - ] - ] [ - - ] ] - ] [ - - ] ]_{10} } &\;  % 341 
  { [ [ [ [ [ - - ] [ - - ] ] - ] - ] [ - - ] ]_{11} } &\;  % 342 
  { [ [ [ [ [ [ - - ] - ] - ] - ] - ] [ - - ] ]_{12} } \\ % 343 
  { [ [ [ [ - - ] [ - - ] ] [ [ - - ] - ] ] - ]_{13} } &\;  % 344 
  { [ [ [ [ [ - - ] - ] - ] [ [ - - ] - ] ] - ]_{14} } &\;  % 345 
  { [ [ [ [ [ - - ] - ] [ - - ] ] [ - - ] ] - ]_{15} } \\ % 346 
  { [ [ [ [ [ - - ] [ - - ] ] - ] [ - - ] ] - ]_{16} } &\;  % 347 
  { [ [ [ [ [ [ - - ] - ] - ] - ] [ - - ] ] - ]_{17} } &\;  % 348 
  { [ [ [ [ [ - - ] - ] [ [ - - ] - ] ] - ] - ]_{18} } \\ % 349 
  { [ [ [ [ [ - - ] [ - - ] ] [ - - ] ] - ] - ]_{19} } &\;  % 350 
  { [ [ [ [ [ [ - - ] - ] - ] [ - - ] ] - ] - ]_{20} } &\;  % 351 
  { [ [ [ [ [ [ - - ] - ] [ - - ] ] - ] - ] - ]_{21} } \\  % 352 
  { [ [ [ [ [ [ - - ] [ - - ] ] - ] - ] - ] - ]_{22} } &\;  % 353 
  { [ [ [ [ [ [ [ - - ] - ] - ] - ] - ] - ] - ]_{23} }    % 354 
  \end{array}
  \]
  \caption{Binary association types in degree 8}
  \label{binary8}
  \[
  \begin{array}{rlr||rlr||rlr} 
  t & \sigma & \varepsilon(\sigma) &
  t & \sigma & \varepsilon(\sigma) &
  t & \sigma & \varepsilon(\sigma) \\
  \midrule
   1  &(12)              &  -1  &   %   1 
   1  &(34)              &  -1  &   %   2 
   1  &(13)(24)          &   1  \\  %   3 
   1  &(56)              &  -1  &   %   4 
   1  &(78)              &  -1  &   %   5 
   1  &(57)(68)          &   1  \\  %   6 
   1  &(15)(26)(37)(48)  &   1  &   %   7 
   2  &(12)              &  -1  &   %   8 
   2  &(34)              &  -1  \\  %   9 
   2  &(13)(24)          &   1  &   %  10 
   2  &(56)              &  -1  &   %  11 
   3  &(12)              &  -1  \\  %  12 
   3  &(56)              &  -1  &   %  13 
   3  &(15)(26)(37)(48)  &   1  &   %  14 
   4  &(12)              &  -1  \\  %  15 
   4  &(45)              &  -1  &   %  16 
   4  &(67)              &  -1  &   %  17 
   5  &(12)              &  -1  \\  %  18 
   5  &(34)              &  -1  &   %  19 
   5  &(13)(24)          &   1  &   %  20 
   5  &(67)              &  -1  \\  %  21 
   6  &(12)              &  -1  &   %  22 
   6  &(67)              &  -1  &   %  23 
   7  &(12)              &  -1  \\  %  24 
   7  &(45)              &  -1  &   %  25 
   7  &(14)(25)(36)      &  -1  &   %  26 
   7  &(78)              &  -1  \\  %  27 
   8  &(12)              &  -1  &   %  28 
   8  &(34)              &  -1  &   %  29 
   8  &(13)(24)          &   1  \\  %  30 
   8  &(56)              &  -1  &   %  31 
   8  &(78)              &  -1  &   %  32 
   9  &(12)              &  -1  \\  %  33 
   9  &(56)              &  -1  &   %  34 
   9  &(78)              &  -1  &   %  35 
  10  &(12)              &  -1  \\  %  36 
  10  &(45)              &  -1  &   %  37 
  10  &(78)              &  -1  &   %  38 
  11  &(12)              &  -1  \\  %  39 
  11  &(34)              &  -1  &   %  40 
  11  &(13)(24)          &   1  &   %  41 
  11  &(78)              &  -1  \\  %  42 
  12  &(12)              &  -1  &   %  43 
  12  &(78)              &  -1  &   %  44 
  13  &(12)              &  -1  \\  %  45 
  13  &(34)              &  -1  &   %  46 
  13  &(13)(24)          &   1  &   %  47 
  13  &(56)              &  -1  \\  %  48 
  14  &(12)              &  -1  &   %  49 
  14  &(56)              &  -1  &   %  50 
  15  &(12)              &  -1  \\  %  51 
  15  &(45)              &  -1  &   %  52 
  15  &(67)              &  -1  &   %  53 
  16  &(12)              &  -1  \\  %  54 
  16  &(34)              &  -1  &   %  55 
  16  &(13)(24)          &   1  &   %  56 
  16  &(67)              &  -1  \\  %  57 
  17  &(12)              &  -1  &   %  58 
  17  &(67)              &  -1  &   %  59 
  18  &(12)              &  -1  \\  %  60 
  18  &(45)              &  -1  &   %  61 
  18  &(14)(25)(36)      &  -1  &   %  62 
  19  &(12)              &  -1  \\  %  63 
  19  &(34)              &  -1  &   %  64 
  19  &(13)(24)          &   1  &   %  65 
  19  &(56)              &  -1  \\  %  66 
  20  &(12)              &  -1  &   %  67 
  20  &(56)              &  -1  &   %  68 
  21  &(12)              &  -1  \\  %  69 
  21  &(45)              &  -1  &   %  70 
  22  &(12)              &  -1  &   %  71 
  22  &(34)              &  -1  \\  %  72 
  22  &(13)(24)          &   1  &   %  73 
  23  &(12)              &  -1  &   %  74 
  \end{array}
  \]
  \caption{Skew-symmetries of binary types in degree 8}
  \label{skew8}
  \end{table}

\smallskip

We apply this lifting process to identities \eqref{LY3}-\eqref{LY6}, which we write as follows:
  \begin{align*}
  f( a, b, c ) 
  &= 
  [[a,b],c] - [[a,c],b] + [[b,c],a] + \langle a, b, c \rangle - \langle a, c, b \rangle + \langle b, c, a \rangle,
  \\
  g_1( a, b, c, d ) 
  &=
  \langle [a,b], c, d \rangle - \langle [a,c], b, d \rangle + \langle [b,c], a, d \rangle,
  \\
  g_2( a, b, c, d )
  &=
  \langle a, b, [c,d] \rangle - [ \langle a, b, c \rangle, d ] + [ \langle a, b, d \rangle, c ],
  \\
  h( a, b, c, d, e )
  &=
  \langle a, b, \langle c, d, e \rangle \rangle 
  - \langle \langle a, b, c \rangle, d, e \rangle 
  + \langle \langle a, b, d \rangle, c, e \rangle 
  - \langle c, d, \langle a, b, e \rangle \rangle.
  \end{align*}
In degree 3, we have only one identity:
  \[
  f( a, b, c ) \equiv 0.
  \]
In degree 4, we have $g_1$, $g_2$ and four consequences of $f$ using $[-,-]$:
  \begin{alignat*}{3}
  &
  g_1( a, b, c, d ) \equiv 0,
  &\qquad &
  g_2( a, b, c, d ) \equiv 0,
  &\qquad &
  f( [a,d], b, c ) \equiv 0,
  \\
  &
  f( a, [b,d], c ) \equiv 0,
  &\qquad &
  f( a, b, [c,d] ) \equiv 0,
  &\qquad &
  [ f( a, b, c ), d ] \equiv 0.
  \end{alignat*}
In degree 5, we have $h$, five consequences of each identity in degree 4 using $[-,-]$, and five consequences of $f$ using
$\langle -, -, - \rangle$; see Table \ref{lifting5}, where we omit ``$\equiv 0$'' to save space.
These 36 identities generate the $S_5$-module of identities in degree 5 satisfied by Lie-Yamaguti algebras.
Some of these identities are redundant, but this issue will not concern us until degree 6.
For example, we do not need to include both 
  \[
  f([[a,e],d],b,c), 
  \qquad
  f([a,[d,e]],b,c) = -f([[d,e],a],b,c),
  \]
in our set of generators, since they differ only by the transposition $(ad)$ and a sign.

  \begin{table}
  \[
  \begin{array}{llll}
  { h( a, b, c, d, e ), }
  &\quad
  { g_1( [a,e], b, c, d ), } 
  &\quad
  { g_1( a, [b,e], c, d ), } 
  &\quad
  { g_1( a, b, [c,e], d ), } 
  \\
  { g_1( a, b, c, [d,e] ), } 
  &\quad
  { [ g_1( a, b, c, d ), e ], }
  &\quad
  { g_2( [a,e], b, c, d ), } 
  &\quad
  { g_2( a, [b,e], c, d ), } 
  \\
  { g_2( a, b, [c,e], d ), } 
  &\quad
  { g_2( a, b, c, [d,e] ), } 
  &\quad
  { [ g_2( a, b, c, d ), e ], }
  &\quad
  { f( [[a,e],d], b, c ), } 
  \\
  { f( [a,[d,e]], b, c ), } 
  &\quad
  { f( [a,d], [b,e], c ), } 
  &\quad
  { f( [a,d], b, [c,e] ), } 
  &\quad
  { [ f( [a,d], b, c ), e ], }
  \\
  { f( [a,e], [b,d], c ), }
  &\quad
  { f( a, [[b,e],d], c ), }
  &\quad
  { f( a, [b,[d,e]], c ), }
  &\quad
  { f( a, [b,d], [c,e] ), }
  \\
  { [ f( a, [b,d], c ), e ], }
  &\quad  
  { f( [a,e], b, [c,d] ), } 
  &\quad
  { f( a, [b,e], [c,d] ), } 
  &\quad
  { f( a, b, [[c,e],d] ), } 
  \\
  { f( a, b, [c,[d,e]] ), } 
  &\quad
  { [ f( a, b, [c,d] ), e ], }
  &\quad
  { [ f( [a,e], b, c ), d ], } 
  &\quad
  { [ f( a, [b,e], c ), d ], } 
  \\
  { [ f( a, b, [c,e] ), d ], } 
  &\quad
  { [ f( a, b, c ), [d,e] ], } 
  &\quad
  { [ [ f( a, b, c ), d ], e ], }
  &\quad
  { f( \langle a, d, e \rangle, b, c ), } 
  \\
  { f( a, \langle b, d, e \rangle, c ), } 
  &\quad
  { f( a, b, \langle c, d, e \rangle ), } 
  &\quad
  { \langle f( a, b, c ), d, e \rangle, } 
  &\quad
  { \langle d, e, f( a, b, c ) \rangle. }
  \end{array}
  \]
  \smallskip
  \caption{Lie-Yamaguti identities in degree 5}
  \label{lifting5}
  \end{table}
  
\begin{definition}
We write $LY(n) \subset BT(n)$ for the $S_n$-submodule of all multilinear polynomial identities in degree $n$ satisfied by 
Lie-Yamaguti algebras.
\end{definition}

Each of the 36 generators of $LY(5)$ produces 6 identities in degree 6 using $[-,-]$, and each of the six generators of
$LY(4)$ produces 6 identities in degree 6 using $\langle-,-,-\rangle$, giving a total of 252 generators for the $S_6$-module 
$LY(6)$.
Continuing the lifting process, we obtain $( 252 + 36 ) \cdot 7 = 2016$ generators for the $S_7$-module $LY(7)$, and
$( 2016 + 252 ) \cdot 8 = 18144$ generators for the $S_8$-module $LY(8)$.

\begin{lemma}
The number $\lambda(n)$ of multilinear Lie-Yamaguti identities in degree $n$ obtained by the lifting procedure 
described above equals
  \[
  \lambda(n) = \frac{1}{20} (n{+}1)! \quad (n \ge 4).
  \]
\end{lemma}

\begin{proof}
The sequence $\lambda(n)$ satisfies $\lambda(4) = 6$, $\lambda(5) = 36$, and the recurrence relation
$\lambda(n) = n \big( \lambda(n{-}1) + \lambda(n{-}2 \big)$, which has the indicated solution.
\end{proof}

We will discuss in the next section how to eliminate redundancies in these sets of Lie-Yamaguti identities 
in order to reduce the size of the computations.

\begin{lemma}
The number of multilinear monomials forming a basis of $BT(n)$ is
  \[
  \mu(n) = \sum_{i=1}^m \frac{n!}{2^{s(T_i)}},
  \]
where $T_1, \dots, T_m$ is a complete list of the association types in degree $n$.
\end{lemma}

\begin{proof}
If $T$ is an association type in degree $n$ with $s(T)$ skew-symmetries, 
then the number of distinct multilinear monomials with type $T$ is $n!/2^{s(T)}$.
\end{proof}

\subsection*{The matrix of Lie-Yamaguti identities}

In principle, for each $n$ we construct a matrix of size $\lambda(n) \times \mu(n)$
in which the $(i,j)$ entry is the coefficient of the $j$-th multilinear monomial in the $i$-th Lie-Yamaguti identity.
We then compute the row canonical form (RCF) of this matrix, and identify those rows whose leading 1s occur
in a column labelled by a monomial in a binary association type.
(Recall that we have sorted the association types so that the binary types correspond to the rightmost columns
of the matrix.
This idea of using a convenient ordering of the association types was introduced by Correa, Hentzel and Peresi
\cite{CHP}.)
Let $b(n)$ be the number of columns corresponding to the binary types, and let $a(n)$ be the number of rows of the RCF
whose leading 1s occur in the last $b(n)$ columns.
Since we have already eliminated all skew-symmetries of the association types in our enumeration of the multilinear 
monomials, it follows that if $a(n) > 0$ then the corresponding rows are polynomial identities satisfied by the bilinear
operation in every Lie-Yamaguti algebra which are not consequences of identities \eqref{LY1} and \eqref{LY2}.

\subsection*{Representation theory of the symmetric group}

We have
  \[ 
  \mu(5) = 300, \qquad
  \mu(6) = 5310, \qquad
  \mu(7) = 109620, \qquad
  \mu(8) = 2751840.
  \]
In order to reduce the size of the matrices, we use the representation theory of the symmetric group.
A theoretical and algorithmic exposition of this method has been given by Bremner and Peresi \cite{BremnerPeresi};
here we summarize the basic ideas.

The irreducible representations of $S_n$ are in one-to-one correspondence with the partitions $\pi$ of $n$.
For $\pi = (n_1,\dots,n_k)$, $n_1 \ge \cdots \ge n_k \ge 1$, $n_1 + \cdots + n_k = n$, we write $d_\pi$
for the dimension of the corresponding representation.
Over any field $F$ of characteristic 0 or $p > n$, the group algebra $FS_n$ is semisimple and has the following
decomposition into the direct sum of simple two-sided ideals, where $M_d(F)$ denotes the $d \times d$ matrices
over $F$:
  \[
  F S_n \cong \bigoplus_\pi M_{d_\pi}(F).
  \]
We write $R_\pi\colon F S_n \to M_{d_\pi}(F)$ for the projection corresponding to $\pi$.
The matrices $R_\pi(\sigma)$ for $\sigma \in S_n$ can be computed using the method of Clifton \cite{Clifton}.
Let $T_1, \dots, T_m$ be the BT association types in degree $n$, and
let $I_1, \dots, I_\ell$ be the Lie-Yamaguti identities in degree $n$.
We collect the terms of each identity by association type:
  \[
  I_i = \sum_{j=1}^m I_{ij},
  \]
where $I_{ij} \in F S_n$ contains the terms of the $i$-th identity with the $j$-th type.
Thus each identity can be written as an element of the direct sum of $m$ copies of $F S_n$.

For each partition $\pi$, we write $L_\pi$ for the $\ell d_\pi \times m d_\pi$ matrix consisting of $d_\pi \times d_\pi$ 
blocks in which the $(i,j)$ block contains the representation matrix $R_\pi(I_{ij})$.
We compute $\mathrm{RCF}(L_\pi)$ and extract the $a_\pi \times b_n d_\pi$ submatrix $A_\pi$
containing the nonzero rows whose leading 1s occur in one of the $d_\pi \times d_\pi$
blocks corresponding to the $b_n$ binary types.
This submatrix represents the consequences of the Lie-Yamaguti identities for partition $\pi$ which involve only the
bilinear operation:
\smallskip
\[
\mathrm{RCF}( L_\pi ) 
=
\begin{array}{cc}
\overbrace{\qquad\qquad\qquad\qquad\qquad}^{\text{ternary and mixed BT types}} 
&\!\!\!\!\!\!
\overbrace{\qquad\qquad}^{\text{binary types}}
\\
\left[ 
\begin{array}{c|}
\qquad \ast \quad \cdots\cdots \quad \ast \qquad
\\
\midrule
\qquad 0 \quad \cdots\cdots \quad 0 \qquad 
\end{array}
\right.
&\!\!\!\!\!\!
\left.
\begin{array}{c}
\ast \; \cdots \; \ast \;\;
\\
\midrule
A_\pi
\end{array}
\right]
\end{array}
\]
\smallskip

We need to exclude the identities which are consequences of the skew-symmetries of the binary types.
We construct the $s_n d_\pi \times b_n d_\pi$ matrix $\overline{B}_\pi$ in which the $(i,j)$ block contains the 
matrix $R_\pi( [\iota]_j + [\sigma_i]_j )$ representing the terms of the $i$-th skew-symmetry in the $j$-th binary type.
We compute $\mathrm{RCF}(\overline{B}_\pi)$ and write $B_\pi$ for the $c_\pi \times b_n d_\pi$
submatrix containing the nonzero rows.

If the row space of $A_\pi$ is a subspace of the row space of $B_\pi$, then for partition $\pi$ we see that every identity
for the bilinear operation is a consequence of the skew-symmetries of the binary types.
If the row space of $A_\pi$ is not a subspace of the row space of $B_\pi$, then each row of $A_\pi$ which does not belong
to the row space of $B_\pi$ represents an identity for the bilinear operation which is not a consequence of the 
skew-symmetries of the binary types. 

If computer memory permits, we do these calculations using rational arithmetic.
However, this becomes impractical except in low degrees.
Therefore, we use modular arithmetic with a prime $p$ greater than the degree $n$ of the identities.
By the structure theory of the group algebra $F S_n$, the ranks of the matrices will be the same as they would have been
if we had used rational arithmetic.

%%%%%%%%%%%%%%%%%%%%%%%%%%%%%%%%%%%%%%%%%%%%%%%%%%%%%%%%%%%%%%%%%%%%%%%%%%%%%%

\section{Main Theorem}

We first show that there are no identities of the required form in degrees 6 and~7, and we then study degree 8,
where we find an identity for the bilinear operation.

In degree 6, there are 38 binary-ternary association types, and 252 Lie-Yamaguti identities.
For each partition $\pi$, we construct the $252 d_\pi \times 38 d_\pi$ matrix in which the $(i,j)$ block contains
the representation matrix for the terms with the $j$-th association type in the $i$-th Lie-Yamaguti identity.
We compute the RCF of this matrix and extract the lower right block containing the binary identities.

There are 6 binary types with 15 skew-symmetries.
We construct the $15 d_\pi \times 6 d_\pi$ matrix in which the $(i,j)$ block contains the representation
matrix for the terms with the $j$-th binary type in the $i$-th skew-symmetry, and compute its RCF.

For every partition $\pi$, we find that the row space of the lower right block of the first matrix is a subspace 
of the row space of the second matrix.
Thus every identity in degree 6 satisfied by the bilinear operation is a
consequence of the skew-symmetries of the binary types.

We repeat the same computation, but we reduce the matrix after each Lie-Yamaguti identity in order to 
keep track of those identities which cause the rank to increase for at least one partition.
We find that the 252 identities in degree 6 are consequences of a subset of 48 identities: 
more than 80\% are redundant. 

In degree 7, there are 113 binary-ternary association types, and 2016 Lie-Yamaguti identities.
There are 11 binary types with 30 skew-symmetries.
Following the same procedure as in degree 6, we find that every identity in degree 7 satisfied by the bilinear
operation is a consequence of the skew-symmetries of the binary types.
Furthermore, the 2016 identities in degree 7 are consequences of a subset of 154 identities:
more than 92\% are redundant. 

  \begin{table}
  \[
  \left[
  \begin{array}{rrrrrrrrrrrrrrrrrrrrrrr}
  1 & \cdot & \cdot & \cdot & \cdot & \cdot &\nn \cdot & \cdot &\nn \cdot & \cdot & \cdot & \cdot & \cdot & \cdot & \cdot & \cdot & \cdot & \cdot & \cdot & \cdot &\nn \cdot & \cdot & \cdot \\
  \cdot & 1 & \cdot & \cdot & \cdot & \cdot &\nn \cdot & \cdot &\nn \cdot & \cdot & \cdot & \cdot & \cdot & \cdot & \cdot & \cdot & \cdot & \cdot & \cdot & \cdot &\nn \cdot & \cdot & \cdot \\
  \cdot & \cdot & 1 & \cdot & \cdot & \cdot &\nn \cdot & \cdot &\nn \cdot & \cdot & \cdot & \cdot & \cdot & \cdot & \cdot & \cdot & \cdot & \cdot & \cdot & \cdot &\nn \cdot & \cdot & \cdot \\
  \cdot & \cdot & \cdot & 1 & \cdot & \cdot &\nn -\frac32 & \cdot &\nn -1 & 1 & \cdot & \cdot & \cdot & 2 & \cdot & \cdot & \cdot & 3 & \cdot & 2 &\nn -2 & \cdot & \cdot \\
  \cdot & \cdot & \cdot & \cdot & 1 & \cdot &\nn \cdot & \cdot &\nn \cdot & \cdot & \cdot & \cdot & \cdot & \cdot & \cdot & \cdot & \cdot & \cdot & \cdot & \cdot &\nn \cdot & \cdot & \cdot \\
  \cdot & \cdot & \cdot & \cdot & \cdot & \cdot &\nn \cdot & 1 &\nn \cdot & \cdot & \cdot & \cdot & \cdot & \cdot & \cdot & \cdot & \cdot & \cdot & \cdot & \cdot &\nn \cdot & \cdot & \cdot \\
  \cdot & \cdot & \cdot & \cdot & \cdot & \cdot &\nn \cdot & \cdot &\nn \cdot & \cdot & 1 & \cdot & \cdot & \cdot & \cdot & \cdot & \cdot & \cdot & \cdot & \cdot &\nn \cdot & \cdot & \cdot \\
  \cdot & \cdot & \cdot & \cdot & \cdot & \cdot &\nn \cdot & \cdot &\nn \cdot & \cdot & \cdot & \cdot & 1 & \cdot & \cdot & \cdot & \cdot & \cdot & \cdot & \cdot &\nn \cdot & \cdot & \cdot \\
  \cdot & \cdot & \cdot & \cdot & \cdot & \cdot &\nn \cdot & \cdot &\nn \cdot & \cdot & \cdot & \cdot & \cdot & \cdot & \cdot & 1 & \cdot & \cdot & \cdot & \cdot &\nn \cdot & \cdot & \cdot \\
  \cdot & \cdot & \cdot & \cdot & \cdot & \cdot &\nn \cdot & \cdot &\nn \cdot & \cdot & \cdot & \cdot & \cdot & \cdot & \cdot & \cdot & \cdot & \cdot & 1 & \cdot &\nn \cdot & \cdot & \cdot \\
  \cdot & \cdot & \cdot & \cdot & \cdot & \cdot &\nn \cdot & \cdot &\nn \cdot & \cdot & \cdot & \cdot & \cdot & \cdot & \cdot & \cdot & \cdot & \cdot & \cdot & \cdot &\nn \cdot & 1 & \cdot
  \end{array}
  \right]
  \]
  \caption{RCF of binary Lie-Yamaguti identities for partition $1^8$}
  \label{rcflowerright8}
  \[
  \left[
  \begin{array}{rrrrrrrrrrrrrrrrrrrrrrr}
  1 & \cdot & \cdot & \cdot & \cdot & \cdot & \cdot & \cdot & \cdot & \cdot & \cdot & \cdot & \cdot & \cdot & \cdot & \cdot & \cdot & \cdot & \cdot & \cdot & \cdot & \cdot & \cdot \\
  \cdot & 1 & \cdot & \cdot & \cdot & \cdot & \cdot & \cdot & \cdot & \cdot & \cdot & \cdot & \cdot & \cdot & \cdot & \cdot & \cdot & \cdot & \cdot & \cdot & \cdot & \cdot & \cdot \\
  \cdot & \cdot & 1 & \cdot & \cdot & \cdot & \cdot & \cdot & \cdot & \cdot & \cdot & \cdot & \cdot & \cdot & \cdot & \cdot & \cdot & \cdot & \cdot & \cdot & \cdot & \cdot & \cdot \\
  \cdot & \cdot & \cdot & \cdot & 1 & \cdot & \cdot & \cdot & \cdot & \cdot & \cdot & \cdot & \cdot & \cdot & \cdot & \cdot & \cdot & \cdot & \cdot & \cdot & \cdot & \cdot & \cdot \\
  \cdot & \cdot & \cdot & \cdot & \cdot & \cdot & \cdot & 1 & \cdot & \cdot & \cdot & \cdot & \cdot & \cdot & \cdot & \cdot & \cdot & \cdot & \cdot & \cdot & \cdot & \cdot & \cdot \\
  \cdot & \cdot & \cdot & \cdot & \cdot & \cdot & \cdot & \cdot & \cdot & \cdot & 1 & \cdot & \cdot & \cdot & \cdot & \cdot & \cdot & \cdot & \cdot & \cdot & \cdot & \cdot & \cdot \\
  \cdot & \cdot & \cdot & \cdot & \cdot & \cdot & \cdot & \cdot & \cdot & \cdot & \cdot & \cdot & 1 & \cdot & \cdot & \cdot & \cdot & \cdot & \cdot & \cdot & \cdot & \cdot & \cdot \\
  \cdot & \cdot & \cdot & \cdot & \cdot & \cdot & \cdot & \cdot & \cdot & \cdot & \cdot & \cdot & \cdot & \cdot & \cdot & 1 & \cdot & \cdot & \cdot & \cdot & \cdot & \cdot & \cdot \\
  \cdot & \cdot & \cdot & \cdot & \cdot & \cdot & \cdot & \cdot & \cdot & \cdot & \cdot & \cdot & \cdot & \cdot & \cdot & \cdot & \cdot & \cdot & 1 & \cdot & \cdot & \cdot & \cdot \\
  \cdot & \cdot & \cdot & \cdot & \cdot & \cdot & \cdot & \cdot & \cdot & \cdot & \cdot & \cdot & \cdot & \cdot & \cdot & \cdot & \cdot & \cdot & \cdot & \cdot & \cdot & 1 & \cdot
  \end{array}
  \right]
  \]
  \caption{RCF of binary skew-symmetries for partition $1^8$}
  \label{rcfskew8}
  \end{table}

In degree 8, we first compute the binary liftings of the 154 identities from degree 7 and the ternary liftings of 
the 48 identities from degree 6.
We obtain altogether $( 154 + 48 ) \cdot 8 = 1616$ identities, less than 9\% of the original set of 18144 identities.

Of the 22 irreducible representations of $S_8$, the first 21 contain only identities for the bilinear operation
which are consequences of the skew-symmetries of the binary types.
The last representation (the sign representation of $S_8$) produces an identity 
which is not a consequence of the skew-symmetries. 
For this representation, the corresponding elements of the group algebra $FS_8$ are the alternating sums 
over all permutations of the variables in each association type.
Since this representation is 1-dimensional, the matrices are relatively small, 
and it is possible to do these calculations using rational arithmetic.

The lower right block of the RCF of the matrix representing the Lie-Yamaguti identities consists 
of those rows whose leading 1s occur in a column corresponding to a binary association type;
this matrix has rank 11 and is displayed in Table \ref{rcflowerright8}.

In the RCF of the matrix representing the skew-symmetries of the binary association types, 
the columns which consist entirely of 0 correspond to those types $T$ whose skew-symmetries all have the form 
$[\iota]_T + [\sigma]_T$ where $\sigma$ is an \emph{odd} permutation (see Table \ref{skew8}); 
this matrix has rank 10 and is displayed in Table \ref{rcfskew8}.

Row 4 of the matrix in Table \ref{rcflowerright8} represents an identity satisfied by the bilinear operation 
in every Lie-Yamaguti algebra 
which is not a consequence of anticommutativity.
This identity is stated explicitly in the next theorem.

\begin{theorem}
The following polynomial identity is satisfied by the bilinear operation in every Lie-Yamaguti algebra over a field of 
characteristic 0 or $p > 2$, but is not a consequence of anticommutativity:
\begin{align*}
\sum_{\sigma \in S_8}
\varepsilon(\sigma)
\Big( \,
&
[[[[a,b],c],[d,e]],[[f,g],h]]
-\tfrac32 \,
[[[[a,b],c],[[d,e],f]],[g,h]]
\\[-8pt]
- \,
&
[[[[[a,b],c],d],[e,f]],[g,h]]
+
[[[[[a,b],c],[d,e]],f],[g,h]]
\\[2pt]
+ \, 2 \,
&
[[[[[a,b],c],d],[[e,f],g]],h]
+ 3 \,
[[[[[a,b],c],[[d,e],f]],g],h]
\\[-2pt]
+ \, 2 \,
&
[[[[[[a,b],c],d],[e,f]],g],h]
- 2 \,
[[[[[[a,b],c],[d,e]],f],g],h]
\, \Big)
\equiv 0,
\end{align*}
where $\varepsilon(\sigma)$ is the sign of $\sigma \in S_8$ which is applied to $a,b,c,d,e,f,g,h$.
\end{theorem}

\begin{corollary}
Let $LY(X)$ be the free Lie-Yamaguti algebra on the set $X$ of generators over a field of characteristic 0 or $p > 2$.
Let $BLY(X)$ the subalgebra of $LY(X)$ generated by $X$ using only the bilinear operation.
If $|X| \ge 8$ then $BLY(X)$ is not a free anticommutative algebra.
\end{corollary}

%%%%%%%%%%%%%%%%%%%%%%%%%%%%%%%%%%%%%%%%%%%%%%%%%%%%%%%%%%%%%%%%%%%%%%%%%%%%%%

\end{document}